\documentclass[11pt]{article}
\usepackage{mathrsfs}
\usepackage{amsfonts}
\usepackage{xcolor}
\usepackage{extarrows}
\usepackage{indentfirst,amssymb,amsmath,graphicx,amsthm,amsfonts,mathrsfs,multicol,amscd}
\usepackage{comment}
\newcommand{\R}{\mathbb{R}}

\def\qed{\hfill $\Box$ \smallskip}
\newtheorem{theorem}{Theorem}[section]

\newtheorem{lemma}[theorem]{Lemma}
\newtheorem{proposition}[theorem]{Proposition}

\theoremstyle{definition}

\newtheorem{remark}{Remark}

\begin{document}
\title{\LARGE\bf{Existence and non-existence of extremal functions for inequalities of Bliss-Moser type  } }
\date{}
 \author{Yanyan Guo$^{1}$, Huxiao Luo$^{2}$$\thanks{{\small Corresponding author. E-mail: yanyangcx@126.com (Y. Guo),
 luohuxiao@zjnu.edu.cn (H. Luo), bernhard.ruf@unimi.it (B. Ruf).}}$, Bernhard Ruf
 $^{3}$\\
\small 1 School of Mathematics and Systems Sciences, Wuhan University of Science and Technology, \\
\small Wuhan 430065, China \\
\small 2 Department of Mathematics, Zhejiang Normal University, Jinhua, Zhejiang, 321004, China \\
\small 3 Istituto Lombardo - Accademia di Scienze e Lettere, 20121 Milan, Italy
}\maketitle

\begin{abstract}
		  We study a singular Moser-type functional with logarithmic weight function
$$\int_0^1 e^{\beta_\gamma u^2(s)}ds,~\beta_\gamma=\frac{\log\frac{e}{s}+\gamma\log\log\frac{e}{s}}{s},~\gamma\in[0,1],$$
on $$E:= \left\{u \in H^{1}(0, 1): u(0) = 0, \int_0^1 |u'|^2dx=1\right\}.$$
We prove compactness for arbitrary normalized weakly convergent
sequences when $\gamma<1$, resolving an open problem proposed in
[J. M. do \'{O}, B. Ruf and P. Ubilla, CVPDE 2023]. At the critical value $\gamma=1$, we consider the perturbed functional
\[
I^c(u):=\int_0^1 e^{(\beta_1(s)+c)u(s)^2}\,ds,
\qquad c\in\mathbb R.
\]
Loss of compactness occurs through ``infinitesimal shock" sequences in $E$ which converge uniformly to $0$ and have a blow-up of the derivative in the origin. We determine its sharp concentration level:
\[
\sup_{(u_j)\in\mathrm{NCS}}
\limsup_{j\to\infty}I^c(u_j)=1+e^2,
\qquad c\in\mathbb R,
\]
where $\mathrm{NCS}$ denotes the class of normalized concentrating
sequences in $E$.
Moreover, there exists a unique threshold $c_0\in\mathbb R$
such that $\sup\limits_{u\in E}I^c(u)$ is attained if and only if
$c\ge c_0$. This result describes how the lower
order perturbations $cu^2$
in the exponent govern the transition between concentration
and attainability at the critical level.
\end{abstract}

\section{Introduction}
\setcounter{equation}{0}

The Trudinger--Moser inequality is the limiting case of the Sobolev
embedding. In dimension two, the critical Sobolev exponent becomes
infinite, and the optimal polynomial growth is replaced by exponential
growth. More precisely, if $\Omega\subset\mathbb R^{2}$ is a bounded
domain, then
\begin{equation}\label{MT}
 \sup_{\substack{u\in H_0^1(\Omega)\\
 \|\nabla u\|_{L^2(\Omega)}\leq 1}}
 \int_{\Omega}e^{\alpha u^2}\,dx<+\infty
 \quad\Longleftrightarrow\quad
 \alpha\leq 4\pi .
\end{equation}
This sharp result, originating in the works of Trudinger \cite{Trudinger} and Moser \cite{Moser},
has become a fundamental tool in nonlinear analysis, geometric
analysis, and the theory of elliptic equations with critical
exponential growth; see \cite{AA,FMR,LL,Pohozaev,R2,R3}. The Trudinger--Moser inequality \eqref{MT} has also been extended  to unbounded domains, one can refer to the works cited in \cite{AT,IM,R1}. Furthermore, Trudinger-Moser type inequalities have been generalized to include weights in the exponential integral \cite{AS, CT, Tarsi}.

The critical nature of the inequality is reflected in a loss of
compactness. A normalized sequence may converge weakly to zero while
the corresponding exponential functional does not converge to the
value of the weak limit. The classical Moser sequence describes the
usual concentration mechanism: the sequence develops a singular
profile at one point, whereas its Dirichlet energy concentrates at
smaller and smaller scales. The quantitative relation between the
supremum and the maximal level of concentration plays a decisive role
in the existence theory for extremal functions. In particular,
Carleson--Chang \cite{CSYAC}, Struwe \cite{Struwe},
and Flucher \cite{Flucher1992} showed that an extremal exists whenever
the variational supremum lies strictly above the maximal value carried
by concentrating sequences.

In a recent paper, do \'{O}, Ruf, and Ubilla \cite{DRU}
introduced a different one-dimensional Moser-type functional.
For
\[
 \beta_\gamma(t)
 =
 \frac{\log(e/t)+\gamma\log\log(e/t)}{t},
 \qquad 0<t<1,
\]
they considered
\[
  I_\gamma(u)
 =
 \int_0^1
 \exp\!\big(\beta_\gamma(t)u^2(t)\big)\,dt
\]
on the class
\[
 E
 :=
 \left\{
 u\in H^1(0,1):
 u(0)=0,\
 \int_0^1|u'(t)|^2\,dt= 1
 \right\}.
\]
They obtained the following results, which have been extended to $N\geq2$ in \cite{GLR1}.
\begin{proposition}\label{20260911-pp1} (\cite[Proposition 3.1]{DRU}) There holds
\begin{equation}\label{20251104-eq1} \sup\limits_{v\in E}I_\gamma(v)<\infty\Longleftrightarrow \gamma\leq1.
\end{equation}
Define the infinitesimal Moser sequence
\begin{equation}\label{moser}
m_j(s):=
\left\{
\begin{array}{ll}
\aligned
&j^{\frac{1}{2}}s\ ,\quad 0\leq s\leq\frac{1}{j}, \\
&\frac 1 {j^{\frac{1}{2}}}\ ,\quad \frac{1}{j}\leq s\leq1.
\endaligned
\end{array}
\right.
\end{equation}
Then
    \begin{equation}\label{20260608}
           I_\gamma (m_j)\to
    \left\{
    \begin{array}{lcl}
         +\infty,\ \ &\hbox{if}\ \gamma>1,\\
         c\ge e+1>1=I_1(0),\ \ &\hbox{if}\ \gamma=1,\\
         1=J_\gamma(0),\ \ &\hbox{if}\ \gamma<1.
    \end{array}
    \right.
    \end{equation}
\end{proposition}
\begin{remark}
    When $\gamma=0$, the inequality \eqref{20251104-eq1}  can be regarded as the limiting case of the Bliss inequality \cite{Bliss}, see the relevant discussion in \cite{GLR1,R} for further details. We therefore call  \eqref{20251104-eq1} the Bliss-Moser type inequality.
\end{remark}
\begin{remark}\label{rem2}
By \cite{DRU}, for any $\gamma\in[0,1]$, we also have that  \[\sup\limits_{u\in B}I_\gamma(u)<\infty,\quad
 B
 :=
 \left\{
 u\in H^1(0,1):
 u(0)=0,\
 \int_0^1|u'(t)|^2\,dt\leq 1
 \right\}.
\]
\end{remark}
At the critical value $\gamma=1$, the loss of compactness is caused by
an infinitesimal shock: the functions tend uniformly to zero, but their
derivatives become unbounded in a shrinking neighbourhood of the
origin. This phenomenon is different from the classical Moser
concentration mechanism.

Although the infinitesimal Moser sequence satisfies
$I_\gamma(m_j)\to1$ for $\gamma<1$, this convergence alone does not
establish compactness for general normalized sequences.
The corresponding question was left open in \cite{DRU}: does weak
convergence in $H^1(0,1)$ imply convergence of the functional
$I_\gamma$ in the subcritical logarithmic range $0\le\gamma<1$?
Our first result gives an affirmative answer to this question.

\begin{theorem}\label{Th2}
	Let $\{u_j\}\subset E$ and suppose that
	\[
	u_j\rightharpoonup u
	\qquad\text{weakly in }H^1(0,1).
	\]
	Then
	\[
	I_\gamma(u_j)\longrightarrow I_\gamma(u)
	\qquad\text{for every }0\le\gamma<1.
	\]
\end{theorem}

The proof, given in Section~2, is based on a comparison with
the critical functional $I_1$. Although $I_1$ is not weakly
continuous, its uniform boundedness on $E$, together with the
strict logarithmic gap $1-\gamma>0$, yields
\[
\lim_{r\downarrow0}\,
\sup_{v\in E}
\int_0^r e^{\beta_\gamma(s)v(s)^2}\,ds=0.
\]
This uniform control near the singular endpoint, combined with the
compact embedding of $H^1(0,1)$ into $C([0,1])$, gives the desired
convergence.
Thus, below the critical logarithmic exponent, concentration of
the exponential integrals at the origin is uniformly excluded.
At $\gamma=1$, this uniform control is lost, and the analysis of
the concentration level and the existence of extremal functions
becomes essential. We turn to these critical phenomena next.

In Sec. 3, we present a concentration-compactness alternative principle that applies to the critical functional $I_1$. Via a blow-up prodecure, and in particular by establishing a local estimate under
small amplitude, we obtain the explicit value for the maximal limit of normalized concentrating
sequences (abbreviated as NCS):
$$\sup_{(u_j)\in {NCS}}\limsup\limits_{j\to\infty} I(u_j)=1+e^2.$$

In Sections 4-6, we extend the critical functional by adding a constant $c$ to the
exponential coefficient $\beta(s)=\left(\log\frac{e}{s}+\log\log\frac{e}{s}\right)/s$
$$I^c(v):=\int_0^1 e^{\left(\log\frac{e}{s}+\log\log\frac{e}{s}+c s\right)\frac{v^2(s)}{s}}ds.
$$
We prove that
\begin{theorem}\label{Th3}
$$\sup\limits_{E}I^c(u)<\infty,\quad\forall c\in\mathbb{R}  $$
and
$$L_{\text{\rm conc}}:=\sup_{(u_j)\in {NCS}}\limsup\limits_{j\to\infty} I^c(u_j)=1+e^2,\quad \text{independently~of}~ c\in \mathbb{R}.$$
\end{theorem}

Finally, we prove
\begin{theorem}\label{Th4}
There exists a unique constant $c_0$ such that for any $c \geq c_0$ the supremum
\(
\sup\limits_{u\in E} I^c(u)
\) is attained, while for $c < c_0$ the supremum is not attained.
\end{theorem}

 In the future, we will investigate the exact value of $c_0$ and the open problem on the attainability of extremal functions for the original critical functional $I$.

The novelty of our results lies in three aspects.

First, we resolve the compactness question left open in \cite{DRU}.

Second, we
determine the sharp concentration level at the critical logarithmic
coefficient, rather than merely proving the existence of
concentrating sequences.

Third, we connect this sharp concentration
level with the attainability and non-attainability of perturbed
variational problems. The analysis also reveals that the critical
loss of compactness is governed by a small-amplitude sequence whose
energy remains asymptotically equal to one while its derivative
develops an infinitesimal shock. This provides a precise variational
description of a concentration mechanism that is distinct from the
classical Moser concentration.

A key ingredient in the determination of the sharp concentration level
is a local estimate for small-amplitude functions. Because the weight
\[
\beta(t)=\frac{\log(\frac et)+\log\log(\frac et)}{t}
\]
is singular at the origin, the condition $\|u\|_{\infty}\ll1$ does not
lead to a straightforward pointwise control of the exponential term.
We overcome this difficulty by using logarithmic variables, a
three-region decomposition of the singular interval, and a refined
one-dimensional energy estimate. This yields the sharp bound
\[
\int_0^r\left(e^{\beta(t)u(t)^2}-1\right)\,dt\le e^2
\]
for all sufficiently small-amplitude functions, see Sec. 3 for the full proof. The importance of this
estimate is that it identifies $e^2$ as the maximal excess mass that can
be generated without a genuine high-amplitude bubble. Combined with the
concentration decomposition, it gives the upper bound
$L_{\mathrm{conc}}\le 1+e^2$; a matching concentrating sequence then
proves
\[
L_{\mathrm{conc}}=1+e^2,\quad \forall c\in\mathbb{R}.
\]
Thus, the lemma is not merely a boundedness estimate: it provides the
sharp local mechanism that converts the singular one-dimensional
geometry into the exact concentration constant $e^2$.

Very recently, upon completing the revised version of this study, we notice that these are some significant progress in the fields of extremals for Trudinger-Moser inequality.
Chen et al. \cite{Chen1} found the $L^p$-perturbation can also affect the existence of extremals for
the Trudinger-Moser inequality. They studied the
attainability of
\[
\sup_{\substack{u\in H_0^1(\Omega)\\ \|\nabla u\|_2\leq 1}}
\int_\Omega \left(e^{4\pi u^2}-\lambda |u|^p\right)\,dx
\]
on bounded planar domains, and established sharp threshold results for the additive $L^p$-perturbations.

 Although our work is motivated by the same
general question, the two problems are mathematically distinct. Their perturbation is an additive $L^p$ term in a planar problem,
while we
consider the one-dimensional singular functionals
\(I^c\)
on $E$, with a perturbation of a singular weight in the exponent.
 Our analysis is based on one-dimensional
localization and a local estimate under small amplitude, and yields the sharp concentration
level $L_{\mathrm{conc}}=1+e^2$, and attainability/non-attainability results for exponential
perturbations. The two works therefore address complementary
perturbation phenomena in different Trudinger--Moser settings.

\section{The proof of Theorem \ref{Th2}}
\setcounter{equation}{0}
\noindent
In this section, we prove that the subcritical functionals $I_\gamma$ ($\gamma\in[0,1)$) are compact.

\begin{proof}[Proof of Theorem~\ref{Th2}]
Fix $0\le\gamma<1$. Let $(u_j)\subset E$ satisfy
\[
u_j\rightharpoonup u
\qquad\text{weakly in }H^1(0,1).
\]
The compact embedding
$H^1(0,1)\hookrightarrow C([0,1])$ implies
\[
u_j\longrightarrow u
\qquad\text{uniformly on }[0,1].
\]
Moreover, uniform convergence and weak lower semicontinuity give
\[
u(0)=0,\qquad
\int_0^1|u'(s)|^2\,ds
\le\liminf_{j\to\infty}
\int_0^1|u_j'(s)|^2\,ds=1.
\]

For $v\in E$, the Cauchy--Schwarz inequality gives
\[
0\le \frac{v(s)^2}{s}
\le\int_0^s|v'(t)|^2\,dt\le1,
\qquad 0<s<1.
\]
Fix $0<r<1$ and set
\[
A_{v,r}
:=\left\{s\in(0,r): \frac{v(s)^2}{s}\le\tfrac12\right\},
\qquad
B_{v,r}:=(0,r)\setminus A_{v,r}.
\]

On $A_{v,r}$, since $\log(e/s)\ge1$ and $\gamma\ge0$,  we have
\[
\begin{aligned}
e^{\beta_\gamma(s)v(s)^2}
\le
\exp\!\left(
\frac{\log(e/s)+\gamma\log\log(e/s)}2
\right)=\left(\frac es\right)^{1/2}
 \left(\log\left(\frac es\right)\right)^{\gamma/2}.
\end{aligned}
\]
On $B_{v,r}$, using $v(s)^2/s>\tfrac12$, $\gamma<1$ and
$\log(e/s)\ge\log(e/r)\ge1$, we obtain
\[
\begin{aligned}
e^{\beta_\gamma(s)v(s)^2}
\le e^{\beta_1(s)v(s)^2}\left[\log\left(\frac er\right)\right]^{-(1-\gamma)/2}.
\end{aligned}
\]
Integrating over the two sets and using Proposition~\ref{20260911-pp1}
\[
M:=\sup_{v\in E}I_1(v)<\infty
\] yields
\begin{equation}\label{eq:subcritical-tail}
\begin{aligned}
\sup_{v\in E}
\int_0^r e^{\beta_\gamma(s)v(s)^2}\,ds
&\le
\int_0^r
\left(\frac es\right)^{1/2}
\left[\log\left(\frac es\right)\right]^{\gamma/2}\,ds\\
&\quad
+M\left[\log\left(\frac er\right)\right]^{-(1-\gamma)/2}
=:\omega_\gamma(r).
\end{aligned}
\end{equation}
With the substitution $t=\log(e/s)$, the first term of $\omega_\gamma(r)$ becomes
\[
\int_0^r
\left(\frac es\right)^{1/2}
\left[\log\left(\frac es\right)\right]^{\gamma/2}\,ds
=
e\int_{\log(e/r)}^\infty
e^{-t/2}t^{\gamma/2}\,dt
\longrightarrow0
\qquad(r\downarrow0).
\]
The second term of $\omega_\gamma(r)$ also tends to zero because $1-\gamma>0$.
Consequently,
\[
\omega_\gamma(r)\longrightarrow0
\qquad(r\downarrow0).
\]

For each fixed $r\in(0,1)$, the function $\beta_\gamma$
is continuous and bounded on $[r,1]$. The uniform
convergence $u_j\to u$ therefore implies
\[
e^{\beta_\gamma(s)u_j(s)^2}
\longrightarrow
e^{\beta_\gamma(s)u(s)^2}
\qquad\text{uniformly on }[r,1].
\]
In particular,
\[
\int_r^1
\left|
e^{\beta_\gamma(s)u_j(s)^2}
-e^{\beta_\gamma(s)u(s)^2}
\right|\,ds
\longrightarrow0.
\]

By Fatou's lemma and
\eqref{eq:subcritical-tail},
\[
\begin{aligned}
\int_0^r e^{\beta_\gamma(s)u(s)^2}\,ds
\le
\liminf_{j\to\infty}
\int_0^r e^{\beta_\gamma(s)u_j(s)^2}\,ds
\le\omega_\gamma(r).
\end{aligned}
\]
By Proposition~\ref{20260911-pp1} and Fatou's lemma, we also have
\[
I_\gamma(u)
\le\liminf_{j\to\infty}I_\gamma(u_j)
\le\sup_{v\in E}I_\gamma(v)<\infty.
\]
Combining these estimates, we find
\[
\begin{aligned}
&\int_0^1
\left|
e^{\beta_\gamma(s)u_j(s)^2}
-e^{\beta_\gamma(s)u(s)^2}
\right|\,ds\\
&\qquad\le
\int_0^r e^{\beta_\gamma(s)u_j(s)^2}\,ds
+\int_0^r e^{\beta_\gamma(s)u(s)^2}\,ds\\
&\qquad\quad
+\int_r^1
\left|
e^{\beta_\gamma(s)u_j(s)^2}
-e^{\beta_\gamma(s)u(s)^2}
\right|\,ds\\
&\qquad\le
2\omega_\gamma(r)
+\int_r^1
\left|
e^{\beta_\gamma(s)u_j(s)^2}
-e^{\beta_\gamma(s)u(s)^2}
\right|\,ds.
\end{aligned}
\]
Taking the upper limit as $j\to\infty$, we obtain
\[
\limsup_{j\to\infty}
\int_0^1
\left|
e^{\beta_\gamma(s)u_j(s)^2}
-e^{\beta_\gamma(s)u(s)^2}
\right|\,ds
\le 2\omega_\gamma(r).
\]
Since $\omega_\gamma(r)\to0$ as $r\downarrow0$, we conclude that
\[
e^{\beta_\gamma u_j^2}
\longrightarrow e^{\beta_\gamma u^2}
\qquad\text{strongly in }L^1(0,1).
\]
In particular,
\[
I_\gamma(u_j)\longrightarrow I_\gamma(u).
\]
\end{proof}

\section{The concentration-compactness alternative and the sharp concentration level  $1+e^2$}\label{scl}
\noindent
Denote
$$\beta(t):=\frac{\log\frac{e}{t}+\log\log\frac{e}{t}}{t},\quad I(v)=\int_0^1e^{\beta(t) v^2(t)}dt,\quad v\in E.$$
Firstly,
we prove the concentration-compactness alternative.
\begin{proposition}\label{pro2.2}
 Let $u_j\in E,~u_j\rightharpoonup u$. Then either
$\{u_j\}$ is a NCS
or $u\neq0$ and
$I(u_j)\to I(u)$.
\end{proposition}
\begin{remark}
    Due to the strong singularity of $\beta(t)$ at the point $0$, this proposition cannot be directly proved using the traditional Lions' method \cite[Theorem I.6]{Lions}.
\end{remark}
\begin{proof}
If $u \equiv0$, then $\{u_j\}$ is a normalized concentrating sequence.

Now assume $u\not\equiv 0$. Set
$$\theta=\int_0^1|u'(t)|^2dt>0.$$
By weak convergence and the Hilbert-space norm identity
$$\int_0^1|u'_j-u'|^2dt=1-\theta+o(1).$$
Thus the remainder has strictly subcritical energy since $1- \theta < 1$.

Away from $t = 0$, compactness gives strong convergence. For every $\delta>0$,
$$u_j \to u~strongly~in~C([\delta, 1]).$$
Hence
\begin{equation}\label{20260912-a1}\int_\delta^1 e^{\beta(t) u_j^2}dt\to\int_\delta^1 e^{\beta(t) u^2}dt.\end{equation}

It remains to control the singular endpoint.
 We first prove that there exists $q>1$ such that
\begin{equation}\label{260912-tag1}
 \sup_j\int_0^1 \exp\!\bigl(q\beta(t)u_j^2(t)\bigr)\,dt<\infty.
\end{equation}

We will prove the uniform integrability estimate \eqref{260912-tag1} by using a smooth truncation of the weak
limit in the following steps.

\medskip
\noindent\textit{Step 1.}
Choose $\eta>0$ so that $1-\theta+\eta<1$. Then, for all sufficiently
large $j$,
\begin{equation}\label{20260912-tag5}
 \int_0^1|(u_j-u)'|^2\,dt\leq1-\theta+\eta.
\end{equation}

Choose $\varepsilon>0$ such that
\begin{equation}\label{20260912-tag6}
 (1+\varepsilon)(1-\theta)<1.
\end{equation}
Next choose $p>1$ and $q>1$ sufficiently close to $1$, and then reduce
$\eta$ if necessary, so that
\begin{equation}\label{20260912-tag7}
 pq(1+\varepsilon)(1-\theta+\eta)<1.
 \end{equation}
This is possible because of the strict inequality in \eqref{20260912-tag6}.

\medskip
\noindent\textit{Step 2.}
Choose $v\in C_c^\infty(0,1)$ such that
\begin{equation}\label{20260912-tag8}
 w:=u-v,
 \qquad
 \int_0^1|w'(t)|^2\,dt<\delta,
\end{equation}
where $\delta>0$ will be chosen below. Since $v$ has compact
support in $(0,1)$, there exists $a>0$ such that
\begin{equation}\label{20260912-tag9}
 v(t)=0\qquad(0<t<a).
\end{equation}

Put $x_j:=u_j-u$. Then
\[
 u_j=x_j+w+v.
\]
Using twice the elementary inequality
\[
 (r+s)^2\leq(1+\varepsilon)r^2
 +\left(1+\frac1\varepsilon\right)s^2,
\]
and, for simplicity, $(w+v)^2\leq2w^2+2v^2$, we obtain
\[
 u_j^2\leq(1+\varepsilon)x_j^2+D w^2+D v^2,
 \qquad
 D:=2\left(1+\frac1\varepsilon\right).
\]
Therefore,
\begin{equation}\label{20260912-tag12}
 e^{q\beta u_j^2}
 \leq e^{q(1+\varepsilon)\beta x_j^2}
       e^{qD\beta w^2}
       e^{qD\beta v^2}.
\end{equation}

\medskip
\noindent\textit{Step 3.}
By \eqref{20260912-tag5}, define
\[x_j:=u_j-u,\quad
 z_j:=\frac{x_j}{\sqrt{1-\theta+\eta}}.
\]
Then $\int_0^1|z_j'|^2\,dt\leq1$, and \eqref{20260912-tag7} gives
\[
 pq(1+\varepsilon)\beta x_j^2
 =pq(1+\varepsilon)(1-\theta+\eta)
   \beta z_j^2<\beta z_j^2.
\]
Hence Remark \ref{rem2} implies
\begin{equation}\label{20260912-tag13}
 \sup_j\int_0^1
 e^{pq(1+\varepsilon)\beta(t)x_j^2(t)}\,dt<\infty.
\end{equation}

\medskip
\noindent\textit{Step 4.}
Since $w(0)=0$, Cauchy--Schwarz gives
\[
 |w(t)|^2
 =\left|\int_0^t w'(s)\,ds\right|^2
 \leq t\int_0^t|w'(s)|^2\,ds
 \leq\delta t.
\]
Consequently,
\[
 \beta(t)w^2(t)
 \leq\delta\left(\log\frac et
 +\log\log\frac et\right).
\]
Choose the approximation in \eqref{20260912-tag8} so close that
\begin{equation}\label{20260912-tag16}
 p'qD\delta<1,
\end{equation}
where $p'=p/(p-1)$. Then
\[
 e^{p'qD\beta(t)w^2(t)}
 \leq
 \left(\frac et\right)^{p'qD\delta}
 \left(\log\frac et\right)^{ p'qD\delta}.
\]
The right-hand side is integrable near $t=0$. Indeed, with
$s=\log(e/t)$,
\[
\begin{aligned}
 \int_0^\delta t^{-a}\left(\log\frac et\right)^b\,dt
 &=e^{1-a}\int_{\log(e/\delta)}^\infty e^{-(1-a)s}s^b\,ds<\infty
 \qquad(a<1).
\end{aligned}
\]
Thus \eqref{20260912-tag16} implies
\begin{equation}\label{20260912-tag19}\int_0^1e^{p'qD\beta(t)w^2(t)}\,dt<\infty.
\end{equation}

\medskip
\noindent\textit{Step 5.}
By \eqref{20260912-tag9}, $v=0$ near $0$. On the compact support of $v$, the function
$\beta$ is bounded, and $v$ is bounded. Therefore
\begin{equation}\label{20260912-tag20}\sup_{0<t<1}e^{qD\beta(t)v^2(t)}<\infty.
\end{equation}

Now, applying H{\"o}lder's inequality to \eqref{20260912-tag12}, and using \eqref{20260912-tag13}, \eqref{20260912-tag19}, and \eqref{20260912-tag20}, gives
\[
\begin{aligned}
 \int_0^1e^{q\beta u_j^2}\,dt
 &\leq
 \left\|e^{qD\beta v^2}\right\|_{L^\infty}
 \left(\int_0^1e^{pq(1+\varepsilon)\beta x_j^2}\,dt\right)^{1/p} \\
 &\quad\times
 \left(\int_0^1e^{p'qD\beta w^2}\,dt\right)^{1/p'}.
\end{aligned}
\]
The right-hand side is bounded independently of $j$, which proves \eqref{260912-tag1}.

 From the uniform integrability \eqref{260912-tag1} and H{\"o}lder's inequality, for every $\delta>0$,
\[
\int_0^\delta e^{\beta(t)u_j^2(t)}\,dt
\leq \delta^{1-1/q}
 \left(\int_0^1 e^{q\beta(t)u_j^2(t)}\,dt\right)^{1/q}
\leq C\delta^{1-1/q},
\]
where $C$ is independent of $j$. Hence
\[
 \sup_j\int_0^\delta e^{\beta(t)u_j^2(t)}\,dt
 \longrightarrow0\qquad(\delta\downarrow0).
\]

Combined with \eqref{20260912-a1}, this gives
$$I(u_j)\to I(u).$$

\end{proof}

Next, we investigate the sharp concentration level $$L_{\text{conc}}:=\sup\limits_{(v_j)\in \text{NCS}}\limsup\limits_{j\to\infty}I(v_j),$$ which is an energy threshold.

We first give a local estimate under small amplitude, which will be also used in Sec. 5.
\begin{lemma}\label{taglm3.1}
Fix once and for all
 $r:=\frac12,$
and define
\[
 J_r(u)=J_{1/2}(u):=
 \int_0^r\bigl(e^{\beta(s)u(s)^2}-1\bigr)\,ds.
\]
There exists an absolute constant $m_0>0$ such that, for every $u\in E$
with $\|u\|_{L^\infty(0,1)}\le m_0$,
\[
 J_r(u)\le e^2.
\]
\end{lemma}

\begin{remark}
This is a uniform upper bound for all functions with sufficiently small
amplitude, not merely an upper bound for the limit along concentrating
sequences.
\end{remark}

\begin{proof}
Throughout the proof, $c,C>0$ denote absolute constants, whose values may
change from line to line. The constants $C_1,C_2$ introduced below are
then fixed before any subsequent thresholds are chosen.

\medskip
\noindent\textit{Step 1.}
For $s>0$, set $q(s)=u(s)^2/s$. Since $u(0)=0$,
\[
 0\le q(s)\le\int_0^s|u'(\tau)|^2\,d\tau\le1,
 \qquad \lim_{s\downarrow0}q(s)=0.
\]
Thus $q$, extended by $q(0)=0$, is continuous on $[0,1]$.
Since $u\not\equiv0$, there is $a\in(0,1]$ such that
\[
 q(a)=\max_{[0,1]}q=1-\Delta,
 \qquad 0\le\Delta<1.
\]
Write
\[
 T=\log\frac ea,\qquad p=T+\log T.
\]
Decompose
\[
 u(s)=\frac{u(a)}a\min\{s,a\}+v(s).
\]
The two derivatives are orthogonal in $L^2(0,1)$, because
$v(0)=v(a)=0$. Hence
\begin{equation}\label{eq:sa-remainder-energy}
 \int_0^1|v'(s)|^2\,ds=\Delta.
\end{equation}
For $s\le a$, use
\[
 v(s)=\int_0^a
 \left(\mathbf{1}_{[0,s]}(\tau)-\frac sa\right)v'(\tau)\,d\tau;
\]
for $s\ge a$, use $v(s)=\int_a^s v'(\tau)\,d\tau$.
The Cauchy--Schwarz inequality gives
\[
 |v(s)|^2\le
 \begin{cases}
  \Delta s(1-s/a),&0\le s\le a,\\
  \Delta(s-a),&a\le s\le1.
 \end{cases}
\]
Introduce
\[
 t=\log\frac es,\qquad
 W(t)=\frac{u(e^{1-t})}{\sqrt{e^{1-t}}},\qquad
 z(t)=1-W(t)^2,\qquad t\ge1.
\]
Then
\begin{equation}\label{eq:sa-envelope}
 |W(t)|\le
 \sqrt{1-\Delta}\,e^{-|t-T|/2}
 +\sqrt{\Delta}\,\sqrt{1-e^{-|t-T|}},
 \qquad z(t)\ge\Delta.
\end{equation}
To obtain a useful lower bound for $z$, put $x=1-e^{-|t-T|}$.
The complementary-square identity and
$(A-B)^2\ge A^2/2-B^2$ yield
\[
 \begin{aligned}
 z(t)&\ge
 \left(\sqrt{1-\Delta}\sqrt{x}
       -\sqrt{\Delta}\sqrt{1-x}\right)^2\\
 &\ge \frac{x}{2}-\Delta.
 \end{aligned}
\]
Taking a suitable convex combination of this estimate and $z\ge\Delta$,
and using $1-e^{-h}\ge c\min\{h,1\}$ for $h\ge0$, we obtain
\begin{equation}\label{eq:sa-deficit}
 z(t)\ge c\bigl(\Delta+\min\{|t-T|,1\}\bigr),\qquad t\ge1.
\end{equation}
In particular,
\[
 z(t)\ge c|t-T|\quad\text{if }|t-T|\le1,
 \qquad
 z(t)\ge c\quad\text{if }|t-T|\ge1.
\]

\medskip
\noindent\textit{Step 2. }
Fix a sufficiently small absolute constant $\delta_0\in(0,1/4)$.
In this step and the next, assume
\[
 \Delta\le\delta_0,\qquad T\ge T_0,
\]
where $T_0$ is a sufficiently large absolute constant.
All requirements on $T_0$ below are independent of $u$.
Set $t_0=\log(e/r)=\log(2e)$.
Extend $u|_{[0,r]}$ constantly to $s>r$, and denote the extension by
$\widetilde u$. Define
\[
 w(t)=\frac{\widetilde u(e^{1-t})}{\sqrt{e^{1-t}}},\qquad t\in\mathbb R.
\]
Then $w=W$ on $[t_0,\infty)$ and
\[
 w'(t)=\frac{w(t)}2-\sqrt{s}\,\widetilde u'(s),\qquad s=e^{1-t}.
\]
For $0<\varepsilon<r<R$, direct integration gives
\[
 \begin{aligned}
 &\int_{\log(e/R)}^{\log(e/\varepsilon)}
       \left(|w'|^2+\frac{w^2}{4}\right)\,dt\\
 &\quad=
 \int_\varepsilon^R
 \left[\left(\widetilde u'-\frac{\widetilde u}{2s}\right)^2
             +\frac{\widetilde u^2}{4s^2}\right]\,ds\\
 &\quad=
 \int_\varepsilon^R|\widetilde u'|^2\,ds
 -\frac12\left[\frac{\widetilde u(s)^2}{s}\right]_\varepsilon^R.
 \end{aligned}
\]
The boundary terms tend to zero as $\varepsilon\downarrow0$ and
$R\to\infty$: at zero this follows from $q(\varepsilon)\to0$, and at
infinity from the constant extension. Consequently $w\in H^1(\mathbb R)$
and
\begin{equation}\label{eq:sa-energy}
 H(w):=\int_{\mathbb R}\left(|w'|^2+\frac{w^2}{4}\right)\,dt
 =\int_0^r|u'(s)|^2\,ds\le1.
\end{equation}
For $T_0$ sufficiently large, $a<r$ and the extension preserves the
maximum:
\begin{equation}\label{eq:sa-maximum}
 \max_{\mathbb R}w^2=1-\Delta.
\end{equation}
Indeed, for $s\ge r$ one has
$\widetilde u(s)^2/s\le u(r)^2/r\le1-\Delta$, and equality is attained
at $s=a$.

For $t\ge t_0$, put $P(t)=t+\log t$. The change of variables
$s=e^{1-t}$ gives
\begin{equation}\label{eq:sa-transformed}
 \begin{aligned}
 J_r(u)
 &=e\int_{t_0}^\infty e^{-t}
          \bigl(e^{P(t)w(t)^2}-1\bigr)\,dt\\
 &=e\int_{t_0}^\infty
          \bigl(t e^{-P(t)z(t)}-e^{-t}\bigr)\,dt.
 \end{aligned}
\end{equation}
We split this integral into three regions. Enlarge $T_0$ so that
$T/4>t_0$, $T-1>T/4$, and $T-1\ge T/2$.

\smallskip
\noindent\emph{The central region $I=[T-1,T+1]$.}
For fixed $z\in[0,1]$, let
$\Phi_z(\xi)=\xi e^{-(\xi+\log\xi)z}$. For $\xi\ge1$,
\[
 |\Phi_z'(\xi)|
 =\xi^{-z}e^{-\xi z}|1-(\xi+1)z|
 \le C(1+\xi z)e^{-\xi z}
 \le C e^{-\xi z/2}.
\]
For each $t\in I$, apply the mean value theorem with $z=z(t)$ held
fixed. Using \eqref{eq:sa-deficit} and $\xi\ge T/2$ between $t$ and $T$,
we find
\[
 \left|t e^{-P(t)z(t)}-T e^{-p z(t)}\right|
 \le C|t-T|e^{-cT|t-T|}.
\]
Also,
\[
 |e^{-T}-e^{-t}|\le C e^{-T}|t-T|,\qquad t\in I.
\]
The exact difference of the two integrands is
\[
 \begin{aligned}
 &\bigl(t e^{-P(t)z(t)}-e^{-t}\bigr)
       -T e^{-p}\bigl(e^{p w(t)^2}-1\bigr)\\
 &\qquad=t e^{-P(t)z(t)}-T e^{-p z(t)}+e^{-T}-e^{-t},
 \end{aligned}
\]
where $T e^{-p}=e^{-T}$. Integrating and using
$\int_0^\infty h e^{-cTh}\,dh\le C/T^2$, we obtain
\begin{equation}\label{eq:sa-central}
 e\int_I e^{-t}\bigl(e^{P(t)w(t)^2}-1\bigr)\,dt
 \le eT e^{-p}\int_I\bigl(e^{p w(t)^2}-1\bigr)\,dt
        +\frac{C}{T^2}.
\end{equation}

\smallskip
\noindent\emph{The region $\Omega=\{t\ge T/4:|t-T|>1\}$.}
Here $z(t)\ge c$ by \eqref{eq:sa-deficit}, so
\[
 \begin{aligned}
 e\int_\Omega\bigl(t e^{-P(t)z(t)}-e^{-t}\bigr)\,dt
 &\le C\int_{T/4}^\infty t e^{-ct}\,dt\\
 &\le C(1+T)e^{-cT/4}\le\frac{C}{T^2}.
 \end{aligned}
\]
Here we first discard the negative term and then enlarge the domain of
the resulting nonnegative upper bound.

\smallskip
\noindent\emph{The region $[t_0,T/4]$.}
By \eqref{eq:sa-envelope},
\[
 w(t)^2\le2\Delta+2e^{-3T/4}\le\frac12,
\]
after choosing $\delta_0$ small and $T_0$ large. Since $e^x-1\le xe^x$
for $x\ge0$,
\[
 e^{1-t}\bigl(e^{P(t)w(t)^2}-1\bigr)
 \le e P(t)t^{1/2}e^{-t/2}w(t)^2.
\]
The integral of $P(t)t^{1/2}e^{-t/2}$ over $[t_0,\infty)$ is finite.
Thus the contribution from this region is bounded by
\[
 C\bigl(\Delta+e^{-3T/4}\bigr)
 \le C\Delta+\frac{C}{T^2}.
\]
Adding the three contributions, and using the nonnegativity of
$e^{pw^2}-1$ to enlarge the integral in \eqref{eq:sa-central}, gives
\begin{equation}\label{eq:sa-localization}
 J_r(u)
 \le eT e^{-p}\int_{\mathbb R}\bigl(e^{pw^2}-1\bigr)\,dt
       +C\Delta+\frac{C}{T^2}.
\end{equation}

\medskip
\noindent\textit{Step 3. }
We use Proposition~\ref{20260909-pro1} from the Appendix with the explicit
choices
\[
 B_0=4,\qquad
 A=16e^{-2}+4(2-2\log2).
\]
Thus, whenever $w\in H^1(\mathbb R)$, $H(w)\le1$, and
$B=p\max_{\mathbb R}w^2\ge4$, we may apply the estimate
\eqref{eq:sa-auxiliary} from the Appendix.
The constants $A$ and $B_0$ are absolute and independent of $u$, $p$,
$T$, and $\Delta$.

Put
\[
 d=p\Delta,\qquad B=p(1-\Delta)=p-d.
\]
Since $\Delta\le\delta_0$, one has $B\ge(1-\delta_0)p$.
After increasing $T_0$ so that
\[
 (1-\delta_0)(T_0+\log T_0)\ge4,
\]
the condition $B\ge4=B_0$ is satisfied.
Combining \eqref{eq:sa-energy}, \eqref{eq:sa-localization}, and
\eqref{eq:sa-auxiliary}, and using $e^{-p}(e^B-1)\le e^{-d}$, yields
\[
 J_r(u)
 \le4eT e^{-d}\left(\frac{d+\alpha}{B}+\frac A{B^2}\right)
       +C\Delta+\frac C{T^2}.
\]
Define
\[
 f(d)=4e^{-d}(d+\alpha),\qquad d\ge0.
\]
The identity $B^{-1}=p^{-1}+d/(pB)$ gives
\[
 J_r(u)
 \le\frac{eT}{p}f(d)+C\Delta+\frac C{T^2}+R,
\]
where
\[
 \begin{aligned}
 R&=4eT e^{-d}\left(\frac{d(d+\alpha)}{pB}+\frac A{B^2}\right)\\
  &\le\frac{CT}{p^2}e^{-d}(d^2+d+1)
   \le\frac C p.
 \end{aligned}
\]
Here we used $T\le p$ and the boundedness of
$e^{-d}(d^2+d+1)$ on $[0,\infty)$. Also, $p\le2T$ for $T\ge T_0$,
so the term $C/T^2$ can be absorbed into $C/p$.
Consequently, there are fixed absolute constants $C_1,C_2>0$ such that
\begin{equation}\label{eq:sa-uniform-profile}
 J_r(u)
 \le\frac{eT}{p}f(d)+C_1\Delta+\frac{C_2}{p},
 \qquad \Delta\le\delta_0,\quad T\ge T_0.
\end{equation}
This estimate is uniform in $d$ throughout the admissible range
$0\le d\le\delta_0p$.

Since $0<\alpha<1$ and
\[
 f'(d)=4e^{-d}(1-\alpha-d),
\]
we have
\begin{equation}\label{eq:sa-sharp-maximum}
 \max_{d\ge0}f(d)=f(1-\alpha)=4e^{\alpha-1}=e,
 \qquad \lim_{d\to\infty}f(d)=0.
\end{equation}
Choose an absolute constant $D>0$ such that
\[
 f(d)\le\frac e2\qquad\text{for all }d\ge D.
\]
Next choose $\Delta_*\in(0,\delta_0]$ sufficiently small that
\[
 C_1\Delta_*\le\frac{e^2}{8}.
\]
Finally choose $T_*\ge T_0$ sufficiently large that
\begin{equation}\label{eq:sa-threshold-choice}
 e^2\log T_*\ge C_1D+C_2+e^2,
 \qquad
 \frac{C_2}{T_*+\log T_*}\le\frac{e^2}{8}.
\end{equation}
These choices depend only on the fixed constants above.

We claim that
\begin{equation}\label{eq:sa-peak-threshold}
 \Delta\le\Delta_*,\quad T\ge T_*
 \quad\Longrightarrow\quad J_r(u)<e^2.
\end{equation}
If $d\le D$, then $\Delta=d/p\le D/p$, and
\eqref{eq:sa-uniform-profile}--\eqref{eq:sa-threshold-choice} give
\[
 \begin{aligned}
 J_r(u)
 &\le e^2\frac Tp+\frac{C_1D+C_2}{p}\\
 &=e^2-\frac{e^2\log T-C_1D-C_2}{p}
 <e^2.
 \end{aligned}
\]
In particular, the error $C_1\Delta$ is retained here and is absorbed,
together with $C_2/p$, by the negative logarithmic correction.
If $d>D$, then
\[
 J_r(u)
 \le\frac{e^2}{2}+C_1\Delta_*+\frac{C_2}{p}
 \le\frac{3e^2}{4}<e^2.
\]
This proves \eqref{eq:sa-peak-threshold}.

\medskip
\noindent\textit{Step 4.}
The small-amplitude assumption does not by itself imply that $\Delta$
is small. We therefore also treat $\Delta\ge\Delta_*$.
Define
\[
 G(s)=\left(\frac es\log\frac es\right)^{1-\Delta_*},
 \qquad 0<s\le\frac12.
\]
Since $\Delta_*>0$,
\[
 \int_0^{1/2}G(s)\,ds
 =e\int_{t_0}^\infty t^{1-\Delta_*}e^{-\Delta_*t}\,dt<\infty.
\]
Choose $\rho\in(0,1/2)$ such that
\[
 \int_0^\rho G(s)\,ds\le\frac{e^2}{4}.
\]
Set $b_\rho=\max_{[\rho,1/2]}\beta<\infty$, and choose $m_1>0$ such that
\[
 \left(\frac12-\rho\right)
       \left(e^{b_\rho m_1^2}-1\right)\le\frac{e^2}{4}.
\]
Finally, let
\begin{equation}\label{eq:sa-amplitude-choice}
 m_0=\min\left\{m_1,
       \sqrt{(1-\Delta_*)e^{1-T_*}}\right\}>0.
\end{equation}
Every quantity in this definition is independent of $u$.

Let $u\in E$ satisfy $\|u\|_\infty\le m_0$.
If $\Delta\ge\Delta_*$, then $u(s)^2/s\le1-\Delta_*$ and hence
$e^{\beta(s)u(s)^2}\le G(s)$ for $0<s\le1/2$. Therefore
\[
 \begin{aligned}
 J_r(u)
 &\le\int_0^\rho G(s)\,ds
  +\left(\frac12-\rho\right)\left(e^{b_\rho m_0^2}-1\right)\\
 &\le\frac{e^2}{2}<e^2.
 \end{aligned}
\]
If $\Delta<\Delta_*$, then
\[
 a=\frac{u(a)^2}{1-\Delta}
 \le\frac{m_0^2}{1-\Delta_*}
 \le e^{1-T_*}.
\]
Thus $T=\log(e/a)\ge T_*$, and \eqref{eq:sa-peak-threshold} gives
$J_r(u)<e^2$ once again. This completes the proof.
\end{proof}

Now, let $\{v_j\}$ be a NCS:
$$\|v_j\|_{H^1}=1,\quad v_j\rightharpoonup 0~\text{in}~H^1(0,1).$$
Since $H^1(0,1)\subset C^0([0,1])$ is compact,
thus $$\|v_j\|_\infty\to0.$$
For sufficient large $j$, $\|v_j\|_\infty\leq m_0.$  So, we can use
Lemma \ref{taglm3.1} to get
$$\int_0^{\frac12} e^{\beta(t)v_j^2}dt=\int_0^{\frac12} \left(e^{\beta(t)v_j^2}-1\right)dt+\frac12\leq e^2+\frac12.$$
On $[\frac12, 1]$, $\beta(t)$ is bounded and thus
$$\int_{\frac12}^1 e^{\beta(t)v_j^2}dt\to1-\frac12,\quad \text{as}~j\to\infty.$$
Therefore, for any NCS $\{v_j\}$, we have $\limsup\limits_{j\to\infty}I(v_j)\leq 1+e^2.$

This implies $L_{conc}\leq 1+e^2.$
To show that the upper bound is sharp, we provide a concentrating sequence attaining $1 + e^2$ below.

Let
$$p_*(z)=-2\log \cosh\frac{z}{2}=-2\log \frac{e^{z/2}+e^{-z/2}}{2}.$$
For $R \to\infty$, define
$$h_R^0(r)=e^{\frac{p_*(R(r-R))}{2R}}=\left(\cosh\frac{R(r-R)}{2}\right)^{-1/R}.$$
Choose $A_R > 0$ so that
$h_R(r) = A_Rh^0_R(r)$ satisfies
$$\int_1^\infty\left|h'_R(r)-\frac{1}{2}h_R(r)\right|^2dr=1.$$
Then return to the original variable $t=e^{1-r}$,
\begin{equation}\label{20260916ee22}
v_R(t):=\sqrt{t} h_R\left(\log\frac{e}{t}\right)
\end{equation}
satisfies that $v_R \in E$ and $v_R \rightharpoonup 0$ in $H^1(0, 1).$

The unnormalized energy of $h^0_R$ has the expansion
$$\int_0^\infty\left|(h^0_R)'(r)-\frac{1}{2}h^0_R(r)\right|^2dr=1+\frac{2\log 2 -1}{R}+O\left(\frac{1}{R^2}\right).$$
Hence
$$A_R^2=1+\frac{2\log 2 -1}{R}+O\left(\frac{1}{R^2}\right).$$
Thus, for $z = R(r - R)$,
$$h_R^2\left(R+\frac{z}{R}\right)=1+\frac{p_*(z)-(2\log 2-1)}{R}+o\left(\frac{1}{R}\right).$$
Substituting this into $I$ gives
$$I(v_R)-1\to ee^{-(2\log 2-1)}\int_{\mathbb{R}}e^{p_*(z)}dz=e\cdot\frac{e}{4}\cdot 4=e^2.$$
Therefore $$I(v_R)\to1+e^2.$$
This proves the exact concentration identity
$$L_{conc}=1+e^2.$$

\section{Attainment for large positive values of $c$}
From now on, we consider the perturbed functional
\[
I^c(u)
=\int_0^1 e^{\beta(t)u^2(t)}e^{c u^2(t)}\,dt \ , \ c \in \mathbb R
\]
In this section we prove Theorem \ref{Th3} and the following attainability result.
\begin{theorem}\label{20261006-t1} There exists an explicit (but not optimal) constant $c^*= 4 \log\left(2(1+e^2)\right)$
such that for any $c > c^*$, the supremum $\sup\limits_{u\in E}I^c(u)$
 is attained.
\end{theorem}

Before proving attainment, we first check that the supremum $\sup\limits_{u\in E}I^c(u)$ is finite.

For every \(u\in E\), Cauchy-Schwarz gives
\begin{equation}\label{tag3.3}
|u(t)|=\left|\int_0^t u'(s)\,ds\right|
\le \sqrt t\left(\int_0^t |u'(s)|^2\,ds\right)^{1/2}
\le \sqrt t\le1.
\end{equation}
Therefore
\[
e^{c u^2(t)}\le e^c.
\]
Hence
\[
I^c(u)
=\int_0^1 e^{\beta(t)u^2(t)}e^{c u^2(t)}\,dt
\le e^c I^0(u).
\]
The critical Bliss-Moser inequality in Proposition \ref{20260911-pp1} gives
\[
\sup_{u\in E} I^0(u)<\infty.
\]
Thus for any fixed $c\in(-\infty,+\infty)$
\[
\sup_{u\in E} I^c(u) <\infty.
\]

For the original critical Bliss-Moser functional $I^0$, we have proved that the concentration level is
\[
 L_{\rm conc}=1+e^2.
\]

For every fixed finite $c$, the concentration level for $I^c$ remains equal to $1+e^2$.
Indeed, if $\{u_j\}\subset E$ is a concentrating sequence, then
\[
 \|u_j\|_{L^\infty(0,1)}\longrightarrow0.
\]
Consequently,
\[
 e^{c u_j^2}=1+o(1)
\]
uniformly, and hence
\[
 I^c(u_j)-I^0(u_j)\longrightarrow0.
\]
It follows that
\[
 L_{\mathrm{conc}}(c):=\sup\limits_{v_j\in E,~v_j~\text{is~NCS}}\limsup\limits_{j\to\infty}I^c(v_j)=1+e^2
\]
for every fixed $c\in\mathbb R$.

Therefore, Theorem \ref{Th3} holds.

To complete the proof of Theorem~\ref{20261006-t1}. We first give the concentration--compactness alternative for $I^c$, which is a corollary of Proposition~\ref{pro2.2}.
\begin{proposition}\label{cor:cc-perturbed}
Let $c\in\mathbb{R}$ be fixed, and let $\{u_j\}\subset E$
satisfy
\[
u_j\rightharpoonup u
\quad\text{weakly in }H^1(0,1).
\]
Then either $\{u_j\}$ is a NCS, or $u\not\equiv0$ and
\[
I^c(u_j)\longrightarrow I^c(u).
\]
\end{proposition}

\begin{proof}
Assume that $\{u_j\}$ is not a NCS.
By Proposition~\ref{pro2.2}, we have
\[
u\not\equiv0,
\qquad
I_1(u_j)\longrightarrow I_1(u).
\]
Moreover, the compact embedding
$H^1(0,1)\hookrightarrow C([0,1])$ gives
$u_j\to u$ uniformly.

Set
\[
f_j(t)=e^{\beta(t)u_j^2(t)},
\qquad
f(t)=e^{\beta(t)u^2(t)}.
\]
Then $f_j\to f$ pointwise on $(0,1)$ and
$\int_0^1 f_j\,dt\to\int_0^1 f\,dt$.
Since these functions are nonnegative, Scheff\'e's lemma
yields
\[
\|f_j-f\|_{L^1(0,1)}\longrightarrow0.
\]
On the other hand, $e^{cu_j^2}\to e^{cu^2}$ uniformly.
Since $\|u_j\|_\infty\le1$, it follows that
\[
\begin{aligned}
|I^c(u_j)-I^c(u)|
&\le
e^{|c|}\|f_j-f\|_{L^1(0,1)}
\\
&\quad+
\|e^{cu_j^2}-e^{cu^2}\|_{L^\infty(0,1)}
\|f\|_{L^1(0,1)}
\longrightarrow0.
\end{aligned}
\]
This proves the proposition.
\end{proof}

Now, we consider the fixed function
$v(t)=t.$
Since
\[
 \int_0^1|v'(t)|^2\,dt=1,
\]
we have $v\in E$. Moreover,
\[
 I^c(v)=\int_0^1e^{(\beta(t)+c)t^2}\,dt
 \longrightarrow+\infty
 \qquad\text{as }c\to+\infty.
\]
Thus there exists $c^*$ such that
\[
 I^c(v)>1+e^2
 \qquad\text{for every }c>c^*.
\]

For example, for $\beta\geq0$, then
\[
 I^c(v)
 \geq \int_{1/2}^1e^{ct^2}\,dt
 \geq \frac12e^{c/4}.
\]
Therefore the explicit, although nonoptimal, condition
\[
 c>4\log\bigl(2(1+e^2)\bigr)
\]
is sufficient to ensure that $I^c(v)>1+e^2$.

Let
\[
 M(c):=\sup_{u\in E}I^c(u).
\]
Assume that $c>c^*$, so that $M(c)>1+e^2$.
Let $\{u_j\}\subset E$ be a maximizing sequence.
Passing to a subsequence, we may assume that
\[
u_j\rightharpoonup U
\quad\text{weakly in }H^1(0,1).
\]
By the concentration--compactness alternative for $I^c$ (Proposition \ref{cor:cc-perturbed})
and the strict inequality $M(c)>1+e^2$, we have
\[
U\not\equiv0,
\qquad
I^c(U)=\lim_{j\to\infty}I^c(u_j)=M(c).
\]

Set $a=\int_0^1|U'|^2\,dt$.
Weak lower semicontinuity gives $0<a\le1$.
If $a<1$, then $V=U/\sqrt a\in E$.
Since $\beta+c>0$ and $U\not\equiv0$, we obtain
\[
I^c(V)>I^c(U)=M(c),
\]
contrary to the definition of $M(c)$.
Thus $a=1$, so $U\in E$ and $U$ attains $M(c)$.
This completes the proof of Theorem~\ref{20261006-t1}.

\section{Non-attainment for large negative values of $c$}
\begin{theorem}\label{20261006-t2} There exists a positive constant $c_*$ such that for any $c < -c_*$, the supremum
$\sup\limits_{u\in E}I^c(u)$ is unattainable.
\end{theorem}
Write $c=-K$, where $K>0$. Then
\[
 I^{-K}(u)=\int_0^1e^{(\beta(t)-K)u(t)^2}\,dt,\quad \beta(t)=\frac{\log\frac{e}{t}+\log\log\frac{e}{t}}{t}.
\]
In this section, we prove Theorem \ref{20261006-t2}. It suffices to prove the following theorem.
\begin{theorem}\label{20260908-th1}
There is a finite constant $K_0>0$ such that, for every $K\ge K_0$,
\begin{equation}
\sup_{u\in E} I^{-K}(u) = L, \quad \text{and} ~ I^{-K}(u) < L ~ \text{for every fixed } u \in E.
\end{equation}
Consequently, the supremum is not attained. Here
$L:=1+e^2$.
\end{theorem}
In what follows, we prove $I^{-K}(u) < L ~ \text{for every fixed } u \in E$.
We divide this assertion into Lemma \ref{20260908-lm0} and Lemma \ref{20260908-lm1} for the proof.

Let $m_0$ be given by Lemma~\ref{taglm3.1} and fix
\[
\delta=\min\left\{\frac14,\frac{m_0^2}{4}\right\}.
\]
\begin{lemma}\label{20260908-lm0}
Define
\[
 \mathcal A_\delta:=
 \left\{u\in E:
 \int_0^\delta|u'(t)|^2\,dt\leq1-\delta
 \right\}.
\]
There exists a finite constant $\rho\in(0,\delta]$ such that
\begin{equation}\label{20260908-e1}
u\in\mathcal A_\delta,~K\geq \beta(\rho)
 \quad\Longrightarrow\quad
 I^{-K}(u)< L.
\end{equation}
\end{lemma}
\begin{proof}
Since
\[
G_\delta(s):=\left(\frac es\log\frac es\right)^{1-\delta}
\in L^1(0,1),
\]
choose $\rho\in(0,\delta]$ such that
\begin{equation}\label{eq:negative-rho}
\int_0^\rho G_\delta(s)\,ds\le\frac12.
\end{equation}

Let $u \in\mathcal A_\delta$ and $K \geq K_0$. When $0 < s < \rho$,
\[
\frac{u(s)^2}{s} \leq \int_0^s |u'(\tau)|^2 d\tau \leq 1 - \delta.
\]
Hence
\begin{equation}\label{20260908-e54}
\int_0^{\rho} e^{(\beta(s)-K)u(s)^2} ds \leq \int_0^{\rho} G_{\delta}(s) ds \leq 1/2.
\end{equation}

And on $[\rho, 1]$, $\beta(s) \leq \beta(\rho) \leq K$, so
\[
e^{(\beta(s)-K)u(s)^2} \leq 1.
\]
Combining with \eqref{20260908-e54}, we obtain
\begin{equation}
I^{-K}(u) \leq 1/2 + 1 - \rho < 3/2 < L.
\end{equation}
Thus, \eqref{20260908-e1} holds.
\end{proof}

Next, we consider functions close to concentration. We use Lemma \ref{taglm3.1}
to establish a refined Carleson--Chang-type estimate.
\begin{lemma}\label{20260908-lm1}
\[K\geq \beta(1/2),\quad
 \int_0^{\delta}|u'(t)|^2\,dt>1-\delta
 \quad\Longrightarrow\quad
 I^{-K}(u)< L
\]
for every $u\in E$.
\end{lemma}

\begin{proof}
Define
\begin{equation*}
C_\delta = \bigl\{ u \in E : \int_0^\delta |u'(s)|^2 ds > 1 - \delta \bigr\}.
\end{equation*}

For $u \in C_\delta$, there is the basic amplitude estimate
\begin{equation*}
\|u\|_\infty \leq 2\sqrt{\delta}.
\end{equation*}
Indeed, for $0 \leq s \leq \delta$, by \eqref{tag3.3},
\[
|u(s)| \leq \sqrt{s} \leq \sqrt{\delta}.
\]
And for $\delta \leq s \leq 1$, from $\int_\delta^1 |u'|^2 < \delta$,
\[
|u(s)| \leq |u(\delta)| + \sqrt{s-\delta} \Bigl(\int_\delta^s |u'(\tau)|^2 d\tau\Bigr)^{1/2}
\leq 2\sqrt{\delta}.
\]

Fix
\(r=1/2\).
Note that
\begin{equation*}
\delta = \min\left\{\frac14,\frac{m_0^2}{4}\right\}<r=\frac{1}{2},
\end{equation*}
where $m_0$ is from Lemma \ref{taglm3.1}. Under this, we have $\|u\|_\infty\leq m_0$. Then we can use Lemma \ref{taglm3.1} to get that for all $u \in C_{\delta}$,
\begin{equation}\label{20260916-eee1}
J_r(u)= \int_0^r \bigl[e^{\beta(s)u(s)^2} - 1\bigr] ds\leq e^2.
\end{equation}

Let
\begin{equation*}
A_r(u) := \int_r^1 \bigl[e^{\beta(s)u(s)^2} - 1\bigr] ds.
\end{equation*}
Obviously, both $J_r(u)$ and $A_r(u)$ are nonnegative, and we have the exact decomposition
\begin{equation}\label{tag2.4}
I^0(u) = 1 + J_r(u) + A_r(u).
\end{equation}
Substituting \eqref{20260916-eee1} into the exact decomposition \eqref{tag2.4}, we obtain
\begin{equation*}
I^0(u) \leq L + A_r(u).
\end{equation*}

Define the perturbation loss
\begin{equation}\label{O3}
I^0(u) - I^{-K}(u)= \int_0^1 e^{\beta(s)u(s)^2} \bigl(1 - e^{-K u(s)^2}\bigr) ds := D_K(u) \qquad\text{for every }u\in E.
\end{equation}

Since $\beta(s)$ decreases in $(0,1]$, for all $s \in [r, 1]$,
$\beta(s) \leq \beta(r).$
Take $K \geq \beta(r).$
Then on $[r, 1]$, $\beta(s) - K \leq 0$, so
\[
e^{(\beta(s)-K)u(s)^2} \leq 1.
\]
Hence, pointwise, we have
\begin{equation*}
\begin{aligned}
e^{\beta(s)u(s)^2} \bigl(1 - e^{-K u(s)^2}\bigr)
&= e^{\beta(s)u(s)^2} - e^{(\beta(s)-K)u(s)^2} \\
&\geq e^{\beta(s)u(s)^2} - 1.
\end{aligned}
\end{equation*}
Integrating over $s \in [r, 1]$ and using the nonnegativity of the perturbation loss on $(0,r)$, we obtain
\begin{equation}\label{tag2.11}
A_r(u) \leq \int_r^1 e^{\beta(s)u(s)^2} \bigl(1 - e^{-K u(s)^2}\bigr) ds \leq D_K(u).
\end{equation}

More precisely, directly rearranging from the definitions gives
\begin{equation*}
\begin{aligned}
D_K(u) - A_r(u) &= \int_0^r e^{\beta(s)u(s)^2} \bigl(1 - e^{-K u(s)^2}\bigr) ds \\
&\quad + \int_r^1 \bigl[1 - e^{(\beta(s)-K)u(s)^2}\bigr] ds.
\end{aligned}
\end{equation*}
Both terms are nonnegative.
Moreover, since $\delta < r$ and there is positive derivative energy on $(0,\delta)$, $u$ cannot be identically zero on $(0,r)$.
When $K > 0$ and $u$ is not identically zero on $(0,r)$, by continuity the set where $u \neq 0$ has positive measure in that interval, so the first term is strictly positive. Therefore, when $K \geq \beta(r) > 0$,
\begin{equation}\label{tag2.13}
A_r(u) < D_K(u).
\end{equation}

Therefore
\begin{equation*}
\begin{aligned}
I^{-K}(u) &= I^0(u) - D_K(u) \\
&\leq L + A_r(u) - D_K(u) \\
&< L.
\end{aligned}
\end{equation*}

That is,
\begin{equation*}
u \in C_{\delta},\; K \geq \beta(r) \quad \Rightarrow \quad I^{-K}(u) < L.
\end{equation*}

This completes the proof of Lemma \ref{20260908-lm1}.
\end{proof}

Now, combining Lemma \ref{20260908-lm0} and Lemma \ref{20260908-lm1}, we have proved
\begin{equation}\label{tag56}
I^{-K}(u) < L~\text{for~every}~u \in E,~K\geq \max\{\beta(\rho),\beta(1/2)\}:=K_0.
\end{equation}

In Sec. 3, we find a sequence $\{v_R\}\subset E$ \eqref{20260916ee22} such that $v_R \rightharpoonup 0$ in $H^1(0, 1),$ and $$I^0(v_R)\to L,\quad\text{as}~R\to\infty.$$

From $v_R \rightharpoonup 0$ in $H^1(0, 1)$,
Since $H^1(0,1)\subset C^0([0,1])$ is compact,
thus $$\|v_R\|_\infty\to0,\quad\text{as} ~R\to\infty.$$
And then
\begin{equation}
I^0(v_R) - I^{-K}(v_R)= \int_0^1 e^{\beta(s)v_R(s)^2} \bigl(1 - e^{-K v_R(s)^2}\bigr) ds \leq \bigl(1 - e^{-K \|v_R\|_\infty^2}\bigr) \int_0^1 e^{\beta(s)v_R(s)^2}  ds\to0,
\end{equation}
as $R\to\infty$.

Therefore
$$\lim\limits_{R\to\infty} I^{-K}(v_R)=L.$$
This together with \eqref{tag56} implies that:
there exists a finite constant $K_0$ such that
   \begin{equation}\label{tag6.23}
    \sup\limits_{u\in E}I^{-K}(u)=L,\quad\text{for~every~}K \geq K_0.
    \end{equation}

\noindent{\it Proof of Theorem \ref{20260908-th1}.}
Suppose that the supremum for some $K \geq K_0$ is attained by $u_* \in E$. By \eqref{tag6.23},
\[
I^{-K}(u_*) = L.
\]
But \eqref{tag56} holds for every fixed $u \in E$, so $I^{-K}(u_*) < L$, a contradiction. Therefore, the supremum is not attained, and Theorem \ref{20260908-th1} is proved.
\qed

\section{The critical perturbation threshold}
Theorems \ref{20261006-t1} and \ref{20261006-t2} show that the attainability and non-attainability phenomena occur
in two different regimes of the perturbation parameter $c$. More precisely, sufficiently positive
perturbations raise the variational level above the concentration threshold, which leads to the
existence of an extremal function, whereas sufficiently negative perturbations preserve the concentration level and prevent the attainment of the supremum.

It is natural to ask whether there exists a sharp transition value separating these two regimes.
Since the functional $I^c$
is monotone with respect to the parameter $c$, namely,
\[
c_1<c_2
\quad\Longrightarrow\quad
I^{c_1}(u)\le I^{c_2}(u),
\qquad \forall\,u\in E,
\]
the attainability property is also monotone in the parameter. Therefore, the two regimes cannot
alternate, and there exists a unique threshold parameter separating attainability from non-attainability. This leads to the following result:
\begin{theorem}\label{20261006-t3}  There exists a finite constant
$$c_0 \in (-c_*, c^*) =\left(-K_0, 4\log\left(2(1+e^2)\right)\right) $$
such that
\[\sup\limits_{u\in E}I^c(u)\]
is attained if $c > c_0$ and is not attained if $c < c_0$. In particular, the parameter $c_0$ separates the
attainability and non-attainability regimes.
\end{theorem}

Recall that
\[
I^c(u)=\int_0^1
\exp\left((\beta(t)+c)u^2(t)\right)\,dt,
\qquad u\in E .
\]
Define the attainability set
\[
\mathcal A
:=
\left\{
c\in\mathbb R:
\sup_{u\in E} I^c(u)
\ \text{is attained}
\right\}.
\]
By Theorem \ref{20261006-t1}, there exists
$c_1\in\mathbb R$ such that
\[
(c_1,+\infty)\subset \mathcal A ,
\]
while by Theorem \ref{20261006-t2}, there exists
$c_2\in\mathbb R$ such that
\[
(-\infty,c_2)\cap\mathcal A=\emptyset .
\]
Hence $\mathcal A$ is non-empty and bounded from below.

Define
\[
c_0:=\inf\mathcal A .
\]
Then $c_0\in\mathbb R$. By the monotonicity of the family
$\{I^c\}_{c\in\mathbb R}$ with respect to $c$, the value $c_0$
serves as the transition point between the attainability and
non-attainability regimes. We now prove this assertion.

\begin{proof}[Proof of Theorem\, \ref{20261006-t3}]
By the definition of $c_0$,
for every $c<c_0$ we have
\[
c\notin\mathcal A,
\]
and hence the supremum is not attained.

It remains to prove that every $c>c_0$ belongs to $\mathcal A$.

First suppose that $c_0\in\mathcal A$.
Then there exists $u_0\in E$ such that
\[
I^{c_0}(u_0)
=
\sup_{u\in E}I^{c_0}(u)=M(c_0).
\]
Since the concentration level is
\[
L_{\mathrm{conc}}=1+e^2,
\]
we have
\[
I^{c_0}(u_0)\ge 1+e^2 .
\]
For every $c>c_0$,
\[
I^c(u_0)>I^{c_0}(u_0)
\ge1+e^2 .
\]
Therefore,
\[
M(c)=\sup_{u\in E}I^c(u)>L_{\mathrm{conc}} .
\]
By the concentration--compactness alternative for $I^c$ (Proposition \ref{cor:cc-perturbed}), the supremum is attained.

Now assume that $c_0\notin\mathcal A$.
For every $\delta>0$, by the definition of the infimum,
there exists
\[
c_\delta>c_0,
\qquad
c_\delta-c_0<\delta ,
\]
such that
\[
c_\delta\in\mathcal A .
\]
Let $u_\delta$ be an extremal function for
$I^{c_\delta}$.
Then
\[
I^{c_\delta}(u_\delta)
=
\sup_{u\in E}I^{c_\delta}(u)
\ge1+e^2.
\]
For every $c>c_\delta$,
\[
I^c(u_\delta)>I^{c_\delta}(u_\delta)
\ge1+e^2.
\]
Hence the supremum of $I^c$ is attained.

Finally, if there existed
\[
\bar c>c_0
\]
such that the supremum of $I^{\bar c}$ were not attained,
choose
\[
c_0<c_\delta<\bar c
\]
with $c_\delta\in\mathcal A$.
The previous argument would imply that the supremum for
$\bar c$ is attained, a contradiction.

Therefore,
\[
(c_0,+\infty)\subset\mathcal A .
\]
The theorem follows.
\end{proof}

To complete Theorem \ref{Th4}, it remains to establish attainability at the endpoint $c_0$.
\begin{theorem}\label{20261006t4}
The supremum
\(
\sup_{u\in E} I^{c_0}(u)
\) is attained.
\end{theorem}

We use the following results from \eqref{eq:sa-uniform-profile} in Lemma \ref{taglm3.1}:

Put
\[
 J_{1/2}(u)=\int_0^{1/2}(e^{\beta(s)u(s)^2}-1)\,ds.
\]
Choose $a\in(0,1]$ with
\[
 \frac{u(a)^2}{a}=\max_{0<s\le1}\frac{u(s)^2}{s}=1-\Delta,
 \qquad
 T=\log\frac ea,\quad p=T+\log T,\quad d=p\Delta .
\]
There are fixed constants $\delta_0\in(0,1)$, $T_0>1$ and
$C_1,C_2>0$ such that
\begin{equation}\label{eq:peak-input}
 J_{1/2}(u)\le \frac{eT}{p}f(d)+C_1\Delta+\frac{C_2}{p},
 \qquad
 \Delta\le\delta_0,\quad T\ge T_0,
\end{equation}
where
\[
 \alpha=2-2\log2,\qquad f(d)=4e^{-d}(d+\alpha),\qquad d\ge0.
\]
In particular,
\begin{equation}\label{eq:f}
 \max_{d\ge0}f(d)=e,\qquad \lim_{d\to\infty}f(d)=0.
\end{equation}

The estimate \eqref{eq:peak-input}, rather than only the qualitative
bound $J_{1/2}(u)\le e^2$, is the quantitative input used below.

\begin{lemma}\label{lem:small}
For every $R>0$, there exists $\varepsilon_R>0$ such that
\[
 u\in E,\quad |c|\le R,\quad
 \|u\|_\infty\le\varepsilon_R
 \quad\Longrightarrow\quad I^c(u)<L.
\]
\end{lemma}

\begin{proof}
For $u\in E$, the Cauchy--Schwarz inequality gives
$|u(s)|^2\le s$, so $m:=\|u\|_\infty\le1$.
Since $\beta$ is bounded on $[1/2,1]$,
\[
 \int_{1/2}^1(e^{\beta(s)u(s)^2}-1)\,ds\le C_{\rm t}m^2
\]
with a constant $C_{\rm t}$ independent of $u$.
Also, for $|c|\le R$,
\begin{equation}\label{eq:perturbation}
 |I^c(u)-I^0(u)|
 \le R e^R m^2 I^0(u)
 \le R e^R M(0)m^2.
\end{equation}
Consequently, with $K_R=C_{\rm t}+R e^R M(0)$,
\begin{equation}\label{eq:full-local}
 I^c(u)-1\le J_{1/2}(u)+K_Rm^2.
\end{equation}

The maximum defining $a$ exists because $u(s)^2/s\to0$ as
$s\downarrow0$. Write
\[
 u(s)=\frac{u(a)}a\min\{s,a\}+v(s).
\]
Since $v(0)=v(a)=0$, the derivatives in this decomposition
are orthogonal in $L^2(0,1)$, and hence
\[
 \int_0^1|v'|^2\,ds=\Delta.
\]
It follows that
\begin{equation}\label{eq:amplitude}
 m\le\sqrt a+\sqrt\Delta,\qquad m^2\le2a+2\Delta.
\end{equation}
Increase $T_0$ to at least $4$. For $T\ge T_0$,
$p\le2T$ and
\[
 ap=e^{1-T}(T+\log T)\le2Te^{1-T}\le8e^{-3}<1.
\]
Combining \eqref{eq:peak-input}, \eqref{eq:full-local} and
\eqref{eq:amplitude}, we obtain
\begin{equation}\label{eq:uniform-perturbed}
 I^c(u)-1
 \le\frac{eT}{p}f(d)+A_R\Delta+\frac{B_R}{p},
 \qquad \Delta\le\delta_0,\quad T\ge T_0,
\end{equation}
where $A_R=C_1+2K_R$ and $B_R=C_2+2K_R$.

Choose $D>0$ such that $f(d)\le e/2$ for $d\ge D$.
Next choose $\delta_R\in(0,\delta_0]$ so small that
$A_R\delta_R\le e^2/8$.
Finally choose $T_R\ge T_0$ sufficiently large that
\[
 e^2\log T_R>A_RD+B_R,\qquad
 \frac{B_R}{T_R+\log T_R}\le\frac{e^2}{8}.
\]
If $\Delta\le\delta_R$ and $T\ge T_R$, then, for $d\le D$,
\[
 I^c(u)-1
 \le e^2\frac Tp+\frac{A_RD+B_R}{p}
 =e^2-\frac{e^2\log T-A_RD-B_R}{p}<e^2.
\]
For $d>D$, the same estimate gives
\[
 I^c(u)-1\le\frac{e^2}{2}+A_R\delta_R+\frac{B_R}{p}
 \le\frac{3e^2}{4}<e^2.
\]

It remains to handle $\Delta\ge\delta_R$ uniformly as $m\to0$.
For this class,
\[
 e^{\beta(s)u(s)^2}
 \le G_R(s):=\left(\frac es\log\frac es\right)^{1-\delta_R},
 \qquad G_R\in L^1(0,1).
\]
For any $\rho\in(0,1)$,
\[
 I^0(u)-1
 \le\int_0^\rho G_R(s)\,ds
 +(1-\rho)\bigl(e^{b_\rho m^2}-1\bigr),
 \qquad b_\rho=\max_{[\rho,1]}\beta.
\]
First taking $\rho$ small and then $m$ small shows that
$I^0(u)-1\to0$ uniformly in this class.
Together with \eqref{eq:perturbation}, this gives
$I^c(u)<L$ for all sufficiently small $m$, uniformly for $|c|\le R$.

Choose $\varepsilon_R$ to satisfy this last requirement and
\[
 \varepsilon_R^2\le(1-\delta_R)e^{1-T_R}.
\]
If $\Delta<\delta_R$ and $m\le\varepsilon_R$, then
\[
 a=\frac{u(a)^2}{1-\Delta}
 \le\frac{\varepsilon_R^2}{1-\delta_R}\le e^{1-T_R},
\]
so $T\ge T_R$ and the preceding estimates apply.
This covers all values of $\Delta$ and proves the lemma.
\end{proof}

\noindent{\bf Proof of Theorem \ref{20261006t4}.}
\begin{proof}
For every $u\in E$, $|u|\le1$. Hence
\begin{equation}\label{eq:parameter-continuity}
 e^{-|c-d|}M(d)\le M(c)\le e^{|c-d|}M(d),
 \qquad c,d\in\R.
\end{equation}
Thus $M(c)$ is finite and continuous. It is also nondecreasing.

If $u_n\in E$ converges weakly to zero, then
$\|u_n\|_\infty\to0$ by compact embedding.
By \eqref{eq:perturbation}, for each fixed $c$,
$I^c(u_n)-I^0(u_n)\to0$.
The concentration level therefore remains $L$ for every $c$.
In particular,
\begin{equation}\label{eq:S-lower}
 M(c)\ge L\qquad(c\in\R).
\end{equation}

We record a useful consequence of convexity.
For any nonzero $v\in B$, put $r=\int_0^1|v'|^2\,ds$.
If $r<1$, then $w=v/\sqrt r\in E$ and
\begin{equation}\label{eq:convexity}
 I^c(v)
 =\int_0^1 e^{r(\beta(s)+c)w(s)^2}\,ds
 \le(1-r)+rI^c(w)
 \le1+r(M(c)-1).
\end{equation}
Here we used $e^{rx}\le(1-r)+re^x$ for $x\in\R$ and $0<r<1$.
This inequality does not require $\beta+c$ to be nonnegative.

We next show that $M(c)>L$ implies $c\in\mathcal A$.
Let $u_n\in E$ be a maximizing sequence for $M(c)$ and,
after passing to a subsequence, suppose $u_n\rightharpoonup u$.
The concentration upper bound excludes $u=0$.
The concentration--compactness alternative for $I^c$ (Proposition \ref{cor:cc-perturbed}) gives $I^c(u)=M(c)$.
Since $M(c)>1$, \eqref{eq:convexity} excludes
$\int_0^1|u'|^2<1$.
Thus $u\in E$ and $c\in \mathcal A$.

For $c<c_0$, attainment is impossible by definition.
The preceding implication and \eqref{eq:S-lower} give $M(c)=L$.
For $c>c_0$, choose $d\in \mathcal A$ with $d<c$, and let $v_d\in E$
be an extremal for $M(d)$. Since $v_d\not\equiv0$,
\[
 M(c)\ge I^c(v_d)>I^d(v_d)=M(d)\ge L.
\]
In particular, every $c>c_0$ belongs to $\mathcal A$.
Continuity of $M$ now yields
\begin{equation}\label{eq:endpoint-level}
M(c_0)=L.
\end{equation}

Set $c_n=c_0+1/n$.
Since $M(c_n)>L$, choose $u_n\in E$ with
\[
 L<I^{c_n}(u_n)\le M(c_n).
\]
By continuity of $S$, $I^{c_n}(u_n)\to L$.
Choose $R>|c_0|+1$. Lemma~\ref{lem:small} implies
\[
 \|u_n\|_\infty>\varepsilon_R\qquad\text{for every }n.
\]
After passing to a subsequence,
\[
 u_n\rightharpoonup u_0\ \text{in }H^1(0,1),
 \qquad u_n\to u_0\ \text{in }C([0,1]).
\]
Consequently $\|u_0\|_\infty\ge\varepsilon_R$, so $u_0\ne0$.
By Proposition~\ref{cor:cc-perturbed}, we have
\[
I^{c_0}(u_n)\longrightarrow I^{c_0}(u_0).
\]
Moreover, since $\|u_n\|_\infty\le1$,
\[
e^{-|c_n-c_0|}I^{c_0}(u_n)
\le I^{c_n}(u_n)
\le e^{|c_n-c_0|}I^{c_0}(u_n).
\]
As $c_n\to c_0$, it follows that
\[
I^{c_0}(u_0)
=\lim_{n\to\infty}I^{c_n}(u_n)=L.
\]
Write $r_0=\int_0^1|u_0'|^2\,ds\in(0,1]$.
If $r_0<1$, then \eqref{eq:convexity} and
\eqref{eq:endpoint-level} imply
\[
 L=I^{c_0}(u_0)
 \le1+r_0(M(c_0)-1)
 =1+r_0(L-1)<L,
\]
a contradiction. Hence $r_0=1$, so $u_0\in E$ and
$I^{c_0}(u_0)=M(c_0)=\sup\limits_{u\in E}I^{c_0}(u)=L$.
\end{proof}

\begin{remark}
The theorem proves attainment at the endpoint, but it does not determine
the sign of $c_0$. In particular, $c_0=0$ would imply attainment of the
unperturbed supremum at the concentration level.
\end{remark}

\section{Appendix. Proof of the auxiliary estimate in Lemma \ref{taglm3.1}}

\begin{proposition}\label{20260909-pro1}
There exist absolute constants $A>0$ and $B_0>0$ such that the following
holds. If $w\in H^1(\mathbb R)$, $p>0$,
\[
 H(w):=\int_{\mathbb R}
 \left(|w'(t)|^2+\frac{w(t)^2}{4}\right)\,dt\le1,
\]
and
\[
 M:=\max_{\mathbb R}w^2,\qquad B:=pM\ge B_0,
\]
then
\begin{equation}\label{eq:sa-auxiliary}
 \int_{\mathbb R}\bigl(e^{pw^2}-1\bigr)\,dt
 \le4(e^B-1)
 \left(\frac pB-1+\frac\alpha B+\frac A{B^2}\right),
 \qquad \alpha:=2-2\log2.
\end{equation}
In fact, one may take
\[
B_0=4,
 \qquad A=16e^{-2}+4(2-2\log2).
\]
\end{proposition}

\begin{proof}
Since $H^1(\mathbb R)\hookrightarrow C_0(\mathbb R)$ in one dimension,
we use the continuous representative of $w$. The assumption $B\ge4$
implies $M>0$, and the maximum defining $M$ is attained. Replacing $w$
by $|w|$ does not change either side of \eqref{eq:sa-auxiliary}; hence we
may assume $w\ge0$.

Set
\[
 \lambda:=\frac{M}{4(e^B-1)},
 \qquad
 \Psi(v):=\frac{v^2}{4}-\lambda\bigl(e^{pv^2}-1\bigr),
 \qquad 0\le v\le\sqrt M.
\]
The function
\[
 x\longmapsto\frac{e^{px}-1}{x},\qquad x>0,
\]
extended continuously at $x=0$ by the value $p$, is increasing. Therefore,
for $0\le v^2\le M$,
\[
 \frac{e^{pv^2}-1}{v^2}
 \le \frac{e^{pM}-1}{M}
 =\frac{e^B-1}{M}.
\]
Consequently,
\[
 \Psi(v)\ge0,\qquad 0\le v\le\sqrt M.
\]

Let $t_0$ be a point at which $w(t_0)=\sqrt M$. Using
$a^2+b^2\ge2ab$ with
$a=|w'|$ and $b=\sqrt{\Psi(w)}$, we obtain
\begin{equation}\label{eq:sa-amgm}
 \begin{aligned}
 H(w)
 &=\int_{\mathbb R}\bigl(|w'|^2+\Psi(w)\bigr)\,dt
   +\lambda\int_{\mathbb R}(e^{pw^2}-1)\,dt\\
 &\ge2\int_{\mathbb R}|w'|\sqrt{\Psi(w)}\,dt
   +\lambda\int_{\mathbb R}(e^{pw^2}-1)\,dt.
 \end{aligned}
\end{equation}

For completeness, the total-variation step is as follows. Define
\[
 G(r):=\int_0^r\sqrt{\Psi(v)}\,dv,
 \qquad 0\le r\le\sqrt M.
\]
Because $w\in C_0(\mathbb R)$, $w(t)\to0$ as $t\to\pm\infty$.
The chain rule on each finite interval, followed by passage to the limit,
gives
\[
 \int_{-\infty}^{t_0}|w'|\sqrt{\Psi(w)}\,dt
 \ge G(\sqrt M),\qquad
 \int_{t_0}^{\infty}|w'|\sqrt{\Psi(w)}\,dt
 \ge G(\sqrt M).
\]
Indeed, each side must carry total variation at least from $0$ to
$\sqrt M$; no monotonicity of $w$ is needed. Hence
\begin{equation}\label{eq:sa-variation}
 H(w)\ge4\int_0^{\sqrt M}\sqrt{\Psi(v)}\,dv
 +\lambda\int_{\mathbb R}(e^{pw^2}-1)\,dt.
\end{equation}

Introduce
\[
 \mathcal J(B):=\int_0^1
 \sqrt{1-\frac{e^{Bx}-1}{x(e^B-1)}}\,dx,
\]
where the integrand at $x=0$ is understood by continuity. With
$x=v^2/M$ and $B=pM$, one obtains
\begin{align*}
 4\int_0^{\sqrt M}\sqrt{\Psi(v)}\,dv
 &=M\int_0^1
 \sqrt{1-\frac{e^{Bx}-1}{x(e^B-1)}}\,dx\\
 &=M\mathcal J(B).
\end{align*}
Therefore, since $H(w)\le1$, \eqref{eq:sa-variation} gives
\[
 M\mathcal J(B)+\lambda\int_{\mathbb R}(e^{pw^2}-1)\,dt\le1.
\]
Using $\lambda=M/[4(e^B-1)]$ and $M=B/p$, we conclude that
\begin{equation}\label{eq:sa-reduced}
 \int_{\mathbb R}(e^{pw^2}-1)\,dt
 \le4(e^B-1)\left(\frac pB-\mathcal J(B)\right).
\end{equation}

It remains to estimate $\mathcal J(B)$. Define
\[
 R_B(x):=\frac{e^{Bx}-1}{x(e^B-1)},
 \qquad
 F(y):=1-\sqrt{1-y},\qquad 0\le y\le1.
\]
The function $R_B$ is increasing on $(0,1]$, because
$x\mapsto(e^{Bx}-1)/x$ is increasing. Moreover, $R_B(1)=1$ and
$0\le R_B(x)\le1$. Since
\[
 1-\mathcal J(B)=\int_0^1F(R_B(x))\,dx,
\]
we estimate this integral on two subintervals.

For $0\le x\le1/2$, monotonicity gives
\[
 R_B(x)\le R_B(1/2)
 =\frac{2}{e^{B/2}+1}\le2e^{-B/2}.
\]
Since $F(y)\le y$,
\begin{equation}\label{eq:sa-first-region}
 \int_0^{1/2}F(R_B(x))\,dx\le e^{-B/2}.
\end{equation}

For $1/2\le x\le1$, we first note that
\[
 \frac{e^{Bx}-1}{e^B-1}\le e^{-B(1-x)},
\]
and therefore
\[
 R_B(x)\le\frac{e^{-B(1-x)}}x.
\]
The elementary inequality $-\log x\le2(1-x)$ for $x\in[1/2,1]$
then yields
\[
 R_B(x)\le e^{-(B-2)(1-x)}.
\]
Since $F$ is increasing,
\begin{align}
 \int_{1/2}^1F(R_B(x))\,dx
 &\le\int_{1/2}^1F\left(e^{-(B-2)(1-x)}\right)\,dx\notag\\
 &\le\frac1{B-2}\int_0^\infty F(e^{-y})\,dy.\label{eq:sa-second-region}
\end{align}
The last inequality follows from the change of variables
$y=(B-2)(1-x)$ and enlargement of the resulting interval.
Furthermore, with $v=\sqrt{1-e^{-y}}$,
\[
 \begin{aligned}
 \int_0^\infty F(e^{-y})\,dy
 &=\int_0^1\frac{2v}{1+v}\,dv\\
 &=2-2\log2=\alpha.
 \end{aligned}
\]
Combining \eqref{eq:sa-first-region} and \eqref{eq:sa-second-region},
we obtain
\begin{equation}\label{eq:sa-J-basic}
 1-\mathcal J(B)\le e^{-B/2}+\frac\alpha{B-2}.
\end{equation}

Now $B\ge4$. Since
\[
 \sup_{x>0}x^2e^{-x/2}=16e^{-2},
\]
we have
\[
 e^{-B/2}\le\frac{16e^{-2}}{B^2}.
\]
Also,
\[
 \frac\alpha{B-2}
 =\frac\alpha B+\frac{2\alpha}{B(B-2)}
 \le\frac\alpha B+\frac{4\alpha}{B^2},
\]
because $B-2\ge B/2$. Thus
\begin{equation}\label{eq:sa-J-final}
 1-\mathcal J(B)
 \le\frac\alpha B+\frac{16e^{-2}+4\alpha}{B^2}.
\end{equation}
Set
\[
 A:=16e^{-2}+4\alpha.
\]
Equation \eqref{eq:sa-J-final} implies
\[
 \frac pB-\mathcal J(B)
 \le\frac pB-1+\frac\alpha B+\frac A{B^2}.
\]
Substitution into \eqref{eq:sa-reduced} proves
\eqref{eq:sa-auxiliary}. Hence the proposition holds with
$B_0=4$ and the stated absolute constant $A$.
\end{proof}

\vspace{3mm}\par\noindent
{\bf Conflict of interest:} The authors state no conflict of interest.
\vspace{3mm}\par\noindent
{\bf Data availability statement:} Data sharing not applicable to this article as no datasets were generated or analyzed during the current study.

\vspace{3mm}\par\noindent
{\bf
Declaration of generative AI and AI-assisted technologies in the manuscript preparation process}

The proof of Lemma \ref{taglm3.1} and Proposition \ref{20260909-pro1} in this manuscript was generated with the assistance of GPT-6 Astra.
All core mathematical ideas of this paper are independently completed and strictly verified by the authors. The authors reviewed, revised and thoroughly checked all outputs generated by GPT-6 Astra, corrected all inappropriate expressions and potential logical deviations in the proof process, and take full responsibility for the academic accuracy and rigor of the full content of the published article.

\end{document}